\documentclass[11pt]{amsart}

\usepackage[colorlinks = true,
            linkcolor = red,
            urlcolor  = magenta,
            citecolor = black,
            anchorcolor = blue]{hyperref}
\usepackage{color}
\usepackage[shortalphabetic,initials,nobysame]{amsrefs}
\usepackage{upgreek}
\usepackage{tikz}
\usepackage[applemac]{inputenc} 
\usetikzlibrary{arrows}
\usepackage{xcolor}
\usepackage[all]{xy}
\usepackage{graphicx}
\usepackage{wrapfig}
\usepackage{yfonts}
\usepackage{graphicx}
\usepackage{amssymb}
\usepackage{epstopdf}
\usepackage{mathrsfs}
\usepackage[mathcal]{eucal}
\usepackage{bm}

\renewcommand{\eprint}[1]{\href{https://arxiv.org/abs/#1}{#1}}

\DeclareMathOperator{\Lie}{Lie}

\newcommand{\fg}{\frak{g}}

\BibSpec{article}{%
    +{}  {\PrintAuthors}                {author}
    +{,} { \textit}                     {title}
     +{,} {}                            {note}
    +{.} { }                            {part}
    +{:} { \textit}                     {subtitle}
    +{,} { \PrintContributions}         {contribution}
    +{.} { \PrintPartials}              {partial}
    +{,} { }                            {journal}
    +{}  { \textbf}                     {volume}
    +{}  { \PrintDate}                {date}
    +{,} { \issuetext}                  {number}
    +{,} { \eprintpages}                {pages}
    +{,} { }                            {status}
    +{,} { \DOI}                   {doi}
    +{,} { \eprint}        {eprint}
      +{,} {\publisher}                {publisher}
    +{,} { \address}                   {address}
    +{}  { \parenthesize}               {language}
    +{}  { \PrintTranslation}           {translation}
    +{;} { \PrintReprint}               {reprint}
    +{.} {}                             {transition}
    +{}  {\SentenceSpace \PrintReviews} {review}
}

\newtheorem{Thm}{Theorem}[section]

\newtheorem{Prop}[Thm]{Proposition}

\theoremstyle{definition}
\newtheorem{Def}[Thm]{Definition}
\theoremstyle{remark}

\newtheoremstyle{named}{}{}{\itshape}{}{\bfseries}{.}{.5em}{#1 #3}
\theoremstyle{named}

\def\C{\mathbb{C}}

\def\P{\mathbb{P}}

\def\fb{\mathfrak{b}}
\def\g{\mathfrak{g}}
\def\Frenkel:2013uda{\mathfrak{h}}

\def\cF{\mathcal{F}}

\def\cL{\mathcal{L}}

\def\cV{\mathcal{V}}

\def\bo{\textbf{o}}

\def\=>{\Longrightarrow}

\def\to{\longrightarrow}

\def\o+{\oplus}
\def\bo+{\bigoplus}

\def\<{\langle}
\def\>{\rangle}
\def\({\left(}
\def\){\right)}

\def\^{\wedge}
\def\+{\dagger}

\def\dd[#1,#2]{\frac{d#1}{d#2}}
\def\del[#1,#2]{\frac{\partial #1}{\partial #2}}
\def\over[#1]{\overline{#1}}
\def\vec[#1]{\overrightarrow{#1}}

\makeatletter
\def\mr@ignsp#1 {\ifx\:#1\@empty\else #1\expandafter\mr@ignsp\fi}%
\newcommand{\multiref}[1]{\begingroup
\xdef\mr@no@sparg{\expandafter\mr@ignsp#1 \: }%
\def\mr@comma{}%
\@for\mr@refs:=\mr@no@sparg\do{\mr@comma\def\mr@comma{,}\ref{\mr@refs}}%
\endgroup}
\makeatother

\newcommand{\hypref}[2]{\ifx\href\asklFrenkel:2013udaas #2\else\href{#1}{#2}\fi}

\tikzset{->-/.style={decoration={
  markings,
  mark=at position .5 with {\arrow{latex}}},postaction={decorate}}}
\tikzset{
    >=latex
    }

\newcommand{\nc}{\newcommand}
\nc{\on}{\operatorname}
\nc{\la}{\lambda}
\nc{\wh}{\widehat}
\nc{\ghat}{\wh\g}
\nc{\mb}{\mathbf}

\begin{document}
\title[On Wronskians and $qq$-systems]{On Wronskians and $qq$-systems}

\author[A.M. Zeitlin]{Anton M. Zeitlin}
\address{
          Department of Mathematics, 
          Louisiana State University, 
          Baton Rouge, LA 70803, USA\newline
Email: \href{mailto:zeitlin@lsu.edu}{zeitlin@lsu.edu},\newline
 \href{http://math.lsu.edu/~zeitlin}{http://math.lsu.edu/$\sim$zeitlin}
}

\date{\today}

\numberwithin{equation}{section}

\begin{abstract}
We discuss the $qq$-systems, the functional form of the Bethe ansatz equations for the twisted Gaudin model from a new geometric point of view. We use a concept of $G$-Wronskians, which are certain meromorphic sections of principal $G$-bundles on the projective line. In this context, the $qq$-system, similar to its difference analog, is realized as the relation between generalized minors of the $G$-Wronskian. We explain the link between $G$-Wronskians and twisted $G$-oper connections, which are the traditional source for the $qq$-systems.
\end{abstract}

\maketitle

\setcounter{tocdepth}{1}
\tableofcontents

\section{Introduction}
The impact of quantum integrable models on modern mathematics is enormous.  
The important examples of this kind are the so-called spin chain models \cite{Baxter:1982zz, Korepin_1993, Reshetikhin:2010si}. While many of the algebraic structures observed there found themselves in pure mathematics, in particular in the modern theory of quantum groups, the original method of solution of these models, known as algebraic Bethe ansatz \cite{Takhtajan:1979iv, Reshetikhin:2010si} remained popular mainly in the framework of mathematical physics. The centerpiece of this method, which on the folklore level is a method of diagonalization of a mutually commuting set of operators (transfer-matrices) in finite-dimensional vector space \cite{Korepin_1993, Reshetikhin:2010si} are the resulting algebraic equations, known as {\it Bethe equations} \cite{Bethe:1931hc, Korepin_1993} which at first sight have no particular mathematical meaning.
In recent years a lot of activity has been devoted to finding the geometric context in which these equations appear naturally. In a particular case of XXX/XXZ spin chains \cite{Takhtajan:1979iv, OGIEVETSKY1986360}, the integrable models, based on Yangians ($Y(\mathfrak{g})$)/quantum affine algebras ($U_{\hbar}(\mathfrak{g})$), the Bethe equations emerge as the relations for the quantum equivariant cohomology/K-theory of a certain variety \cite{Pushkar:2016qvw, Koroteev:2017aa, Aganagic:2017be, Okounkov:2015aa} as conjectured in the theoretical physics context \cite{Nekrasov:2009ui, Nekrasov:2009uh}. At the same time, these relations take a straightforward way, written in terms of a system of difference equations, known as $QQ$-systems. Incidentally for $U_{\hbar}(\mathfrak{g})$ the $QQ$-systems emerge as the relations in the extended Grothendieck ring of finite dimensional representations of quantum affine algebras \cite{Bazhanov:1998dq, HJ, Frenkel:2013uda, Frenkel:2016}.

Much earlier, another geometric realization due to B. Feigin, 
E. Frenkel and their collaborators \cite{Feigin:1994in, Frenkel:2003qx, Frenkel:2004qy,Feigin:2006xs ,Rybnikov:2010} was achieved for the semiclassical version of the aforementioned integrable models, the so-called Gaudin model. It turned out the Bethe equations, in this case, describe certain principal $^LG$-bundle connections (group $^LG$ corresponds to the Langlands dual $^L\mathfrak{g}$) on the projective line, called opers, with a prescribed singularity structure. This geometrization of Gaudin Bethe equations was a part of a far bigger story: this correspondence is an example of a geometric Langlands correspondence.

A naturally arising question is whether there exists a deformation of this example if a similar correspondence holds for Bethe equations of XXX/XXZ models, and what is a proper generalization of the principal bundle connection. In \cite{KSZ, Frenkel:2020} we introduced the deformed version of the connection, which we called $(^LG,\hbar)$-opers for Bethe equations of XXZ type \footnote{In the non-simply laced case of $\mathfrak{g}$, the situation is more involved, see recent paper \cite{FHRnew}.}. While \cite{Frenkel:2020} we treated $(^LG,\hbar)$-opers on Lie-theoretic level, in \cite{KSZ ,Koroteev:2020tv} for $(SL(N),\hbar)$-opers we exploited a different approach, which used interpretations of $QQ$-systems as minors in the deformed and twisted version of the Wronskian matrix. At the time, it seemed like a construction specific for defining the representation of $SL(N)$. Still, later, it was observed in \cite{KZminors} that a $QQ$-system emerges from a new object, $(^LG,\hbar)$-Wronskian: a meromorphic section of a principal $^L G$-bundle satisfying a certain difference equation. We established the explicit correspondence between $(^LG,\hbar)$-opers and $(^LG,\hbar)$-Wronskians in \cite{KZminors}, so that the $QQ$-system emerges as relations between generalized minors \cite{FZ ,Fomin:wy} of $(G,\hbar)$-Wronskian.

In \cite{Brinson:2021ww} we described the classical limit of the $QQ$-system: we called it $qq$-system, which is a system of differential equations representing the original Gaudin model/oper context. In this note we explain how to obtain the differential $G$-Wronskian and generalized minor interpretation of the $qq$-system. 

The exposition is as follows. First, Section 2 reviews the concept of generalized minors following Fomin and Zelevinsky. Then, in Section 3, we introduce the idea of $G$-Wronskian for simply connected simple group $G$ and its relation to the $qq$-system. Next, in Section 4, we discuss the class of $G$-opers, which correspond to the $qq$-system, following \cite{Brinson:2021ww}. Finally, in Section 5, we establish the relation between two objects: $G$-opers and $G$-Wronskians and discuss the differences between $G$-Wronskians and their deformed analogs.

\subsection*{Acknowledgements} The author is indebted to E. Frenkel, P. Koroteev, and D. Sage for fruitful discussions. The author is partially supported by Simons Collaboration Grant 578501 and NSF grant DMS-2203823.

\section{Generalized minors}    
\subsection{Group-theoretic data}
Let $G$ be a connected, simply connected, simple algebraic group of
rank $r$ over $\mathbb{C}$.  We fix a Borel subgroup $B_-$ with
unipotent radical $N_-=[B_-,B_-]$ and a maximal torus $H\subset B_-$.
Let $B_+$ be the opposite Borel subgroup containing $H$.  Let $\{
\alpha_1,\dots,\alpha_r \}$ be the set of positive simple roots for
the pair $H\subset B_+$.  Let $\{ \check\alpha_1,\dots,\check\alpha_r
\}$ be the corresponding coroots; the elements of the Cartan matrix of
the Lie algebra $\fg$ of $G$ are given by $a_{ij}=\langle
\alpha_j,\check{\alpha}_i\rangle$. 

The Lie algebra $\fg$ has Chevalley
generators $\{e_i, f_i, \check{\alpha}_i\}_{i=1, \dots, r}$, so that
$\fb_-=\Lie(B_-)$ is generated by the $f_i$'s and the
$\check{\alpha}_i$'s and $\fb_+=\Lie(B_+)$ is generated by the $e_i$'s
and the $\check{\alpha}_i$'s, while $\mathfrak{h}=\mathfrak{b}_{\pm}/[\mathfrak{b}_{\pm}, \mathfrak{b}_{\pm}]$ is generarted by $\check{\alpha}_i$'s.  
Let $\{\omega_i\}_{i=1,\dots,r}$ and $\{\check{\omega}_i\}_{i=1,\dots,r}$ be the
fundamental weights and coweights correspondingly, defined by $\langle \omega_i,
\check{\alpha}_j\rangle=\langle \check{\omega_i},
{\alpha}_j\rangle\delta_{ij}$. The element ${\rm ad}_{\check{\rho}}$, where $\check{\rho}=\sum^r_{i=1}{\check{\omega}_i}$  defines a principal gradation on $\mathfrak{b}_+=\oplus_{i\ge 0}\mathfrak{b}_{+,i}$.

Let $W_G=N(H)/H$ be the Weyl group of $G$. Let $w_i\in W$, $(i=1,
\dots, r)$ denote the simple reflection corresponding to
$\alpha_i$. We also denote by $w_0$ be the longest element of $W$, so
that $B_+=w_0(B_-)$.  We also fix
representatives $s_i\in N(H)$ of $w_i$, and in general will denote $\tilde{w}$ the representative of 
$w$ in $N(H)$.



\subsection{Generalized Minors and their properties}
 Consider the big cell $G_0\in G$ in Bruhat decomposition: $G_0=N_-HN_+$.
Any element $g\in G_0$ can be represented as follows:
\begin{eqnarray}
g=n_-~h~n_+.
\end{eqnarray}
Let $V_i$ be the irreducible representation of $G$ with highest weight $\omega_i$ and highest weight vector $\nu_{i}$ which is the eigenvector for any element $h\in H$, i.e. $h\nu_{i}=[h]^{\omega_i}\nu_{i}$, $[h]^{\omega_i}\in \mathbb{C}^{\times}$. 
Then we can formulate the following definition.
\begin{Def}\cite{FZ}
Regular functions $\{\Delta^{\omega_i}\}_{i=1, \dots, r}$ on $G$, whose values on a dense set $G_0$ are given by
\begin{eqnarray}
\Delta^{\omega_i}(g)=[h]^{\omega_i}, \quad i=1, \dots, r
\end{eqnarray}
will be referred to as  {\it principal minors} of a group element $g$. 
\end{Def}

In case of $G=SL(r+1)$, these functions coincide with the principal minors of the standard matrix realization of $SL(r+1)$. We define other generalized minors using the action of  the lifts of the Weyl group elements on the right and the left and then applying principal minors to the result. In other words, we have the following definition.
\begin{Def}\cite{FZ}
For $u, v\in W_G$, we define a regular function $\Delta_{u\omega_i, v\omega_i}$ on $G$ by setting
\begin{equation}
\Delta_{u\cdot \omega_i, v\cdot \omega_i}(g)=\Delta^{\omega_i}(\tilde{u}^{-1}g~\tilde{v}).
\end{equation}
\end{Def}
 Notice that in this notation $\Delta_{\omega_i, \omega_i}(g)=\Delta^{\omega_i}(g)$. Consider the orbit $\mathcal{O}_{W_G}=W_G\cdot \mathbb{C}\nu^+_{\omega_i}$. Then we have the following Proposition.

\begin{Prop}
Action of the group element on the highest weight vector $\nu_i\in V_i$ is as follows:
\begin{eqnarray}
g\cdot \nu_{i}=\sum_{w\in W}\Delta_{w\cdot \omega_i, \omega_i}(g)\tilde{w}\cdot\nu_{i}+\dots,
\end{eqnarray}
where dots stand for the vectors, which do not belong to the orbit $\mathcal{O}_W$.
\end{Prop}

The set of generalized minors $\{\Delta_{w\cdot \omega_i, \omega_i}\}_{w\in W; i=1, \dots, r}$ creates a set of coordinates on $G/B^+$, known as {\it generalized Pl\"ucker coordinates}. In particular, the set of zeroes of each of $\Delta_{w\cdot \omega_i, \omega_i}$ is a uniquely and unambiguously defined hypersurface in $G/B$. This feature is important for characterizing Schubert cells as quasi-projective subvarieties of a generalized flag variety, see \cite{Fomin:wy} for details. 

We started this section from the explicit definition of the principal minors by means of Gaussian decomposition. The following proposition (see Corollary 2.5 in \cite{FZ}) provides a necessary and sufficient condition of its existence for a given group element.

\begin{Prop}\label{gausscond}
An element $g\in G$ admits the Gaussian decomposition if and only if $\Delta^{\omega_i}(g)\neq 0$ for any $i=1, \dots, r$.
\end{Prop}

Finally, we introduce the fundamental relation (\cite{FZ}, Theorem 1.17) between generalized minors, which will be crucial in the following. 

\begin{Prop}
Let, $u,v\in W$, such that for 
$i\in \{1, \dots, r\}$,  $\ell(uw_i)=\ell(u)+1$,  $\ell(vw_i)=\ell(v)+1$. Then 
\begin{equation}\label{eq:minorsgen}
\Delta_{u\cdot\omega_i, v\cdot\omega_i}(g)\Delta_{uw_i\cdot \omega_i, vw_i\cdot\omega_i}(g)-
\Delta_{uw_i\cdot\omega_i, v\cdot\omega_i}(g)\Delta_{u\cdot\omega_i, vw_i\cdot\omega_i}(g)=\prod_{j\neq i}\Big[\Delta_{u\cdot\omega_j, v\cdot\omega_j}(g)\Big]^{-a_{ji}}~.
\end{equation}
\end{Prop}
\section{$G$-Wronskians}

\subsection{Differential equations and $G$-Wronskian}

Consider the irreducible  representation $V_i$ of $G$ with highest weight
$\omega_i$. It comes equipped with a 2-dimensional subspace $W_i=L^+_i\oplus L^-_i$, 
$L^+_i=\mathbb{C}\nu_{i}$, $L^-_i=\mathbb{C}f_i\nu_{i}$, 
which is invariant under the action of $B_+$. 

Suppose we have a principal $G$-bundle $\cF_G $ and its $B_+$-reduction $\cF_{B_+} $ and thus an $H$-reduction  $\cF_{H}$, where $H=B_+/[B_+,B_+]$ as well.
Therefore for each
$i=1,\ldots,r$, the vector bundle
$$
\cV_i = \cF_{B_+} \underset{B_+}\times V_i = \cF_G \underset{G}\times
V_i
$$
associated to $V_i$ contains an $H$-line subbundles
$$
\cL^+_i = \cF_{H} \underset{H}\times L^+_i, \quad \cL^{-}_{i} = \cF_{H} \underset{H}\times L^{-}_{i}
$$
associated to $L_i^{\pm} \subset V_i$.

\begin{Def}
Generalized $G$-Wronskian on $\mathbb{P}^1$ is the quadruple $(\cF_G,\cF_{B_+}, \mathscr{G}, \nabla^Z)$, where $\mathscr{G}$ is a meromorphic section of a principle bundle $\cF_G$, $\cF_{B_+}$ is a reduction of $\cF_G$ to $B_+$, $\nabla^Z$ is an $H$-connection, so that  $H=B_+/[B_+,B_+]$, satisfying the following condition.
There exist a Zariski open dense subset $U\subset \mathbb{P}^1$ 
together with the trivialization $\imath_{B_+}$ of $\cF_{B_+}$ so that $i_{\frac{d}{dz}}\nabla^Z=\nabla^Z_z=\partial_z-Z$, where 
$Z\in \mathfrak{h}=\mathfrak{b}_+/[\mathfrak{b}_+,\mathfrak{b}_+]$, 
so that 
for certain sections $\{v_i^{\pm}\}_{i=1, \dots, r}$ of line bundles  $\{\mathcal{L}^{\pm}_i\}_{i=1, \dots, r}$ on $U$ we have $\mathscr{G}$ as an element of $G(z)$ satisfy the following conditions:
\begin{eqnarray}
\nabla^Z_z(\mathscr{G}\cdot v^+_{i})=\mathscr{G}\cdot v^-_{i}, \quad  i=1,\dots, r
\end{eqnarray}
\end{Def}

In local terms we have the following. Representing the corresponding section $\mathscr{G}=\mathscr{G}(z)\in G(z)$, we have:
\begin{eqnarray}
(\partial_z-Z)\mathscr{G}(z)\nu_{i}=\mathscr{G}(z)p^{\phi}_{-1}(z)\nu_{i},
\end{eqnarray}
where $p^{\phi}_{-1}(z)=\sum^r_{i=1}\phi_i(z)f_i$, and $\phi_i(z)\in \mathbb{C}(z)$. 
We are interested in the case when $\{\phi_i(z)\}_{i=1,\dots , r}$ are polynomials.

\begin{Def}
Generalized $G$-Wronskian has regular singularities if 
$p^{\phi}_{-1}(z)=p^{\Lambda}_{-1}(z)=\sum^r_{i=1}\Lambda_i(z)f_i$, where 
$\Lambda_i(z)\in \mathbb{C}[z]$.
\end{Def}

We also are interested to impose some non-degeneracy conditions on $\mathscr{G}(z)$, which has to do with their generalized minor structure. 

\begin{Def}
We say that generalized $G$-Wronskian with regular singularities is {\it nondegenerate} if  $\Delta_{w\cdot \omega_i,\omega_i}(\mathscr{G}(z))$ are nonzero polynomials for all $w\in W$ and $i=1,\dots, r$.  
\end{Def}

Now let us look at the equations, satisfied by non-degenerate genralized $G$-Wronskians: 
\begin{eqnarray}\label{eqnreg}
(\partial_z-Z)\mathscr{G}(z)\nu_{i}=\mathscr{G}(z)p^{\Lambda}_{-1}(z)\nu_{i}
\end{eqnarray}

Restricting this equation to $W_i\subset V_i$, we have the following:
\begin{eqnarray}\label{relations}
&&(\partial_z-\langle Z, \omega_i\rangle )\Delta_{\omega_i, \omega_i}(\mathscr{G}(z))=\Lambda_i(z)\Delta_{\omega_i, w_i\cdot\omega_i}(\mathscr{G}(z))\nonumber\\
&&(\partial_z-\langle Z, \omega_i-\alpha_i\rangle )\Delta_{w_i\cdot\omega_i, \omega_i}(\mathscr{G}(z))=\Lambda_i(z)\Delta_{w_i\cdot\omega_i, w_i\cdot\omega_i}(\mathscr{G}(z))
\end{eqnarray}
Let us make use of the relation (\ref{eq:minorsgen}), when $u,v=1$, applying it to $\mathscr{G}(z)$:

\begin{equation}\label{eq:minorsgens}
\Delta_{\omega_i, \omega_i}(\mathscr{G}(z))\Delta_{w_i\cdot \omega_i, w_i\cdot\omega_i}(\mathscr{G}(z))-
\Delta_{w_i\cdot\omega_i, \omega_i}(\mathscr{G}(z))\Delta_{\omega_i, w_i\cdot\omega_i}(\mathscr{G}(z))=\prod_{j\neq i}\Big[\Delta_{\omega_j,\omega_j}(\mathscr{G}(z))\Big]^{-a_{ji}}.
\end{equation}

With the help of (\ref{relations}) we obtain: 
\begin{eqnarray}
&&\Lambda_i^{-1}(z)\Delta_{\omega_i, \omega_i}(\mathscr{G}(z))(\partial_z-\langle Z, \omega_i-\alpha\rangle )\Delta_{w_i\omega_i, \omega_i}(\mathscr{G}(z))-\nonumber\\
&&\Lambda_i^{-1}(z)
\Delta_{w_i\cdot\omega_i, \omega_i}(\mathscr{G}(z))(\partial_z-\langle Z, \omega_i\rangle )\Delta_{\omega_i, \omega_i}(\mathscr{G}(z))=\prod_{j\neq i}\Delta_{\omega_j,\omega_j}^{-a_{ji}}(\mathscr{G}(z)). 
\end{eqnarray}

Denoting $\Delta_{\omega_i, \omega_i}(\mathscr{G}(z))=q^i_+(z)$, $\Delta_{w_i\cdot\omega_i, \omega_i}(\mathscr{G}(z))=q^i_-(z)$, then simplifying, and collecting such equations for all fundamental weights, we arrive to the following system:

\begin{eqnarray}\label{qq}
&&q^i_+(z)\partial_z q^i_-(z)-q^i_-(z)\partial_z q^i_+(z)+\langle Z, \alpha_i\rangle q^i_+(z) q^i_-(z)=\nonumber\\
&&\Lambda_i(z)\prod_{j\neq i}\Big[q_+^j(z)\Big]^{-a_{ji}}, \quad i=1,\dots, r.
\end{eqnarray}

Applying $\tilde{w}\in G$ which is a lift of the element of $w\in W$ to (\ref{eqnreg}), we have: 
\begin{eqnarray}\label{eqnregw}
(\partial_z-Z^w)\mathscr{G}^w(z)\nu_{i}=\mathscr{G}^w(z)p^{\Lambda}_{-1}(z)\nu_{i}~,
\end{eqnarray}
where $Z^w\in \tilde{w}Z\tilde{w}^{-1}$ and $\mathscr{G}^w(z)=\tilde{w}~ \mathscr{G}(z)$. 
Let us denote $\Delta_{w^{-1}\cdot \omega_i, \omega_i}(\mathscr{G}(z))=q^{i,w}_+(z)$. Then, if $\ell({w^{-1}s_i})=\ell({w^{-1}})+1$, we have 
$\Delta_{w^{-1}\cdot \omega_i, \omega_i}(\mathscr{G}(z))=q^{i,w}_+(z)$, so that $\Delta_{s_i\cdot \omega_i, \omega_i}(\mathscr{G}(z))=q^{i}_-(z)$ and the following relations hold:
\begin{eqnarray}\label{qqfull}
&&q^{i,w}_+(z)\partial_z q^{i,w}_-(z)-q^{i,w}_-(z)\partial_z q^{i,w}_+(z)+\langle Z^w, \alpha_i\rangle q^{i,w}_+(z) q^{i,w}_-(z)=\nonumber \\
&&\Lambda_i(z)\prod_{j\neq i}\Big[q_+^{j,w}(z)\Big]^{-a_{ji}}, \quad i=1,\dots, r; \quad w\in W. 
\end{eqnarray}

\begin{Def}
 The collection of equations (\ref{qq}) on polynomials $\{q^i_{\pm}\}_{i=1,\dots ,r}(z)$ is called the {\it $qq$-system}. 
 The collection of equations (\ref{qqfull}) is called {\it the full $qq$-system}.
\end{Def}

Using this definition we can formulate the following Theorem.

\begin{Thm}\label{wrthm}
i) The element $\mathscr{G}(z)\in G(z)$, which defines the nondegenerate $G$-Wronskian has a Gaussian decomposition $\mathscr{G}(z)=\mathscr{B}_-(z)\mathscr{N}_+(z)$. \\
ii) Generalized minors, which determine $\mathscr{B}_-(z)$ are the solutions of the $qq$-system, according to the formula $q^{i,w}_+(z)=\Delta_{w^{-1}\cdot \omega_i, \omega_i}(\mathscr{G}(z))$.\\
iii) Representing $\mathscr{N}_+(z)=e^{{n}_+(z)}$, the restriction $n_+(z)|_{\mathfrak{b}_{+,1}(z)}$ is determined by the solution of the $qq$-system.     \\
iv) Given a solution $\mathscr{G}(z)$ of (\ref{eqnregw}), we obtain that $\mathscr{G}(z)U_+(z)$, where  
$U_{+}(z)\in [N_+,N_+](z)$, i.e. $U_+(z)=e^{u_+(z)}$, so that ${u}_+(z)\in \oplus_{k\ge 2}\mathfrak{n}_{k}$. 
\end{Thm}
\begin{proof}
To prove $i)$, it is enough to refer to nondegeneracy condition of $\mathscr{G}(z)$ and use \ref{gausscond}. 
We derived $ii)$ above. Part $iii)$ follows from the fact that all components of $n_+(z)|_{\mathfrak{b}_{+,1}(z)}$ which is a linear combination of $e_i$'s appear in the right hand side of (\ref{eqnreg}), since they act on $f_i\nu_i$. For the same reason $iv)$ is true since $e^{u_+(z)}|_{W_i}={\rm id}_{W_i}$, if ${u}_+(z)\in \oplus_{k\ge 2}\mathfrak{n}_{k}$.
\end{proof}

Part  iv) of the above Theorem motivates the following definition.

\begin{Def}
Two $G$-Wronskians $\mathscr{G}_1(z)$, $\mathscr{G}_2(z)$ are called {\it equivalent }if $\mathscr{G}_1(z)=\mathscr{G}_2(z)U_+(z)$, where $U_{+}(z)\in [N_+,N_+](z)$.
\end{Def}

Theorem (\ref{wrthm}) implies that there is a map from equivalence classes of nondegenerate $G$-Wronskians to the solution of the full $qq$-systems, where $\{q^{w,i}_+(z)\}_{w\in W, i=1,\dots, r}$ are nonzero. As one could guess, we would like an inverse map. We will construct it in the last section. An important tool for that is the concept 
of $Z$-twisted Miura $G$-oper with regular singularities, which we give a review in the next section.

\section{$Z$-twisted $G$-opers on $\mathbb{P}^1$ and the $qq$-systems}

\subsection{$Z$-twisted $G$-opers on $\mathbb{P}^1$}
We now define meromorphic $G$-oper conenctions, or, simply, $G$-opers, on $\P^1$. 
Let us consider the pair $(\cF_G,\nabla)$ of a principal $G$-bundle on $\P^1$ and a
connection, which is automatically flat.  Let
$\cF_{B_+}$ be a reduction of $\cF_G$ to the Borel subgroup $B_+$.
If $\nabla'$ is any connection which preserves $\cF_{B_+}$, then
$\nabla-\nabla'$ induces a well-defined one-form on $\P^1$ with values
in the associated bundle $(\mathfrak{g}/\mathfrak{b}_+)_{\cF_{B_+}}$. 
We denote this 1-form by $\nabla/\cF_{B_-}$.

Following \cite{Beilinson:2005} (see also \cite{Frenkel_LanglandsLoop}) we will define a $G$-oper as a
$G$-connection $(\cF_G,\nabla)$ together with a  reduction 
$\cF_{B_+}$, such that this reduction is not preserved by the 
connection but instead satisfies a certain condition on the 1-form $\nabla/\cF_{B_+}$.

Let $\bf{O}\in
[\mathfrak{n}_+,
\mathfrak{n}_+]^{\perp}/\mathfrak{b}_+\in\mathfrak{g}/\mathfrak{b}_+$
be the open  $B_+$-orbit consisting of vectors stabilized by $N_+$ and
such that all of the simple root components with
respect to the adjoint action of $B_+/N_+$, are non-zero, where the
orthogonal complement is taken with respect to the Killing form.

\begin{Def}    \label{op}
A meromorphic $G$-{\em oper} on $\mathbb{P}^1$ is a triple $(\cF_G,\nabla,\cF_{B_+})$, 
where pair ($\cF_G, \Delta)$ is a principal $G$-bundle on $\P^1$ with a meromorphic connection and $\mathcal{F}_{B_+}$ is a reduction of $\cF_G$ to $B_+$ satisfying the following condition: there exists a Zariski open dense subset $U \subset \P^1$ together with a trivialization of $\mathcal{F}_{B_+}$ such that the restriction of the 1-form $\nabla/\cF_{B_+}$ to $U$, written as an element of $\mathfrak{g}/\mathfrak{b}_+(z)$, belongs to $\mathbf{O}(z)$.
\end{Def}

Using the trivialization $\imath_{B_-}$, the $G$-oper can be written as a differential operator:
\begin{equation}    \label{op1}
\nabla=\partial_z+\sum^r_{i=1}\phi_i(z)f_i+b(z)
\end{equation}
where $\phi_i(z) \in\C(z)$ and  $b(z)\in \mathfrak{b}_+(z)$ are
regular on $U$ and moreover $\phi_i(z)$ has no zeros in $U$.

One can impose the following restrictions on $\phi_i$:

\begin{Def}
We say that a meromorphic $G$-oper has regular singularities if $\phi_i(z)=\Lambda_i(z)$, where $\Lambda_i(z)\in \mathbb{C}[z]$ for all $i=1,\dots, r$.
\end{Def}

\subsection{$Z$-twisted Miura $G$-opers with regualr singularities}
\begin{Def}    \label{Miura}
  A {\em Miura $G$-oper} on $\mathbb{P}^1$ is a quadruple
 $(\cF_G,\nabla,\cF_{B_+},\cF_{B_-})$, where $(\cF_G,\nabla,\cF_{B_+})$ is a
  meromorphic $G$-oper on $\P^1$ and $\cF_{B_-}$ is a reduction of
  the $G$-bundle $\cF_G$ to $B_-$ that is preserved by the
  connection $\nabla$.
\end{Def}

Let us discuss the relative position of the two reductions
over any $x\in\P^1$.  This relative position will be an element of the
Weyl group.  To define this, first
note that the fiber
$\cF_{G,x}$ of $\cF_G$ at $x$ is a $G$-torsor with reductions
$\cF_{B_-,x}$ and $\cF_{B_+,x}$ to $B_-$ and $B_+$
respectively.   Under this
isomorphism, $\cF_{B_-,x}$ gets identified with $gB_- \subset G$ and
$\cF_{B_+,x}$ with $hB_+$ for some $g,h\in G$. The quotient $g^{-1}h$
specifies an element of
the double coset space $B_-\backslash G/B_+$.  The Bruhat
decomposition gives a bijection
between this spaces and the Weyl group, so we obtain a well-defined
element of $G$. We say that $\cF_{B_-}$ and $\cF_{B_+}$ have  {\em generic
  relative position} at $x\in\P^1$ if the relative position is the
identity element of $W$.  More concretely, this mean that the quotient
$g^{-1}h$ belongs to the open dense Bruhat cell $B_-B_+ \subset
G$. It turns out the following theorem holds.

\begin{Thm}    \label{gen rel pos}
i) For any Miura $G$-oper on $\mathbb{P}^1$, there exists an open dense subset $V \subset \P^1$ 
such that the reductions $\cF_{B_-}$ and $\cF_{B_+}$ are in generic relative position for all $x \in V$.\\
ii) For any Miura $G$-oper with regular singularities on $\mathbb{P}^1$, there exists a
trivialization of the underlying $G$-bundle $\cF_G$ on an open
dense subset of $\P^1$ for which the oper connection has the form
\begin{equation}    \label{genmiura}
\nabla=\partial_z-\sum^r_{i=1}g_i(z)\check{\alpha}_i+\sum^r_{i=1}{\Lambda_i(z)}f_i,
\end{equation}
where $g_i(z), \phi_i(z)\in \mathbb{C}(z)$.
\end{Thm}

Let us impose a strong condition, which picks up a subset of opers we are interested in, namely the  Miura $G$-opers, which are gauge equivalent 
to a constant connection.

\begin{Def}
A {\em $Z$-twisted Miura $G$-oper} on $\mathbb{P}^1$ is a Miura $G$-oper that is equivalent to the constant element $Z \in \mathfrak{b}_- \subset \mathfrak{g}(z)$ under the gauge action of $G(z)$.
\end{Def}

Here we immediately can decompose the twist $Z$ into 
\begin{eqnarray}
Z=Z^H+Z^{N_-}, \quad Z^H=\sum^r_{i=1}\zeta_i\check{\alpha}_i~,
\end{eqnarray}
breaking it into Cartan and nilpotent part. 

For untwisted opers, there is a full flag variety $G/B_-$ of
associated  
Miura opers. If $Z=Z^H$, this space is discrete and is one-to-one correspondence with Weyl group $W$. In, general, in the twisted case, we must introduce certain closed
subvarieties of the flag manifold of the form $(G/B_-)_Z=\{gB_-\mid
g^{-1}Zg\in\fb_-\}$, known as \emph{Springer fibers} (see, for example, Chapter 3 of \cite{CG}).  For
$SL(n)$ (or $GL(n)$),  a Springer fiber may be viewed as the space of complete
flags in $\C^n$ preserved by a fixed endomorphism.

\begin{Prop}\cite{Brinson:2021ww}  The map from Miura
$Z$-twisted opers to $Z$-twisted opers is a fiber bundle with fiber
$(G/B_-)_Z$.
\end{Prop}

Taking the quotient of $\cF_{B_-}$ by
$N_- = [B_-,B_-]$, we obtain
an $H$-bundle $\cF_{B_-}/N_-$ endowed with an 
$H$-connection, which we will refer to as \emph{associated Cartan connection}: $\nabla^H=\partial_z+A^H(z)$, so that 
\begin{equation}    \label{AH}
A^H(z)=\sum^r_{i=1} g_i(z){\check{\alpha}_i}.
\end{equation}

For $Z$-twisted Miura $G$-opers, we immediately obtain that  
\begin{equation}    \label{giyi}
g_i(z)=\zeta_i -y_i(z)^{-1}\partial_zy_i(z),
\end{equation}
where $y_i(z)\in \mathbb{C}(z)$.

\subsection{Nondegenerate $Z$-twisted Miura $G$-opers, $qq$-systems and B\"acklund transformations}

Now we will impose nondegeneracy conditions on $y_i(z)$, which will lead us to the relation between Miura $G$-opers and the $qq$-systems. We will formulate it in the algebraic manner and refer to \cite{Brinson:2021ww} for more geometric formulation.

\begin{Def}\label{nondegop}
A $Z$-twisted Miura $G$-oper is called nondegenerate, if:\\
\noindent i) it has the form 
(\ref{genmiura}) with $g_i(z)$ satisfying (\ref{giyi}), where:
\begin{enumerate}\item $y_i(z)$ are polynomials with no multiple zeros;
    \item if $a_{ik}\ne 0$, then the roots of
      $\Lambda_k(z)$ are distinct from the the zeros and poles of
      $y_i(z)$; and
      \item if $i\ne j$ and there exists $k$ for which $a_{ik},
        a_{jk}\ne 0$, then the zeros and poles of $y_i(z)$ and
  $y_j(z)$ are distinct from each other.
\noindent 
\end{enumerate}
\end{Def}

A related definition, which imposes a similar type of conditions on the solution of the $qq$-system is as follows:

\begin{Def}
A polynomial solution $\{ q^i_+(z),q^i_-(z) \}_{i=1,\ldots,r}$ of
\eqref{qq} is called {\em nondegenerate} if each $q^i_+(z)$ is relatively prime to
$q^i_-(z)$, and the $q^i_+(z)$'s satisfy the conditions in
Definition~\ref{nondegop}. 
\end{Def}


Then the following statement is true.

\begin{Thm} \cite{Brinson:2021ww}
There is one-to-one correspondence between nondegenerate $Z$-twisted Miura $G$-opers and nondegenerate solutions of the $qq$-system, which allows polynomial solutions to the full $qq$-system, so that 
\begin{equation}
y_i(z)=q^i_+(z), \quad i=1, \dots, r.
\end{equation} 
Moreover, any $Z$-twisted Miura $G$-oper is $Z^H$-twisted.
\end{Thm}

One can write explicit algebraic equations on zeroes of $q_+^i(z)=\prod^{{\rm deg}(q^i_+(z)}_{\ell=1}(z-w_\ell^i)$ polynomials (without the loss of generality we can assume $q^i_+(z)$ to be monic). These algebraic equations are known as {\it Bethe equations} for $Z$-twisted $^L\mathfrak{g}$-Gaudin model \cite{Feigin:2006xs}, \cite{Rybnikov:2010}:

\begin{equation}\begin{gathered}\label{bethe}
\langle \alpha_i,Z^H\rangle+\sum^{N}_{j=1}\frac{\langle \alpha_i,\check{\lambda}_j\rangle}{w_\ell^i-z_ j}-\sum_{(j,s) \neq (i,\ell)} \frac{a_{ji}}{w_\ell^i-w_s^j}=0,\\ 
i=1,\dots, r, \quad  \ell=1, \dots, \deg(q^i_+(z)).
\end{gathered}
\end{equation}

The correspondence between $qq$-systems and Bethe ansatz equations is summarized in the  following Proposition.

\begin{Thm}\cite{Brinson:2021ww}\label{BAE}
i) If $Z^H$ is regular, there is a bijection between the solutions of the
Bethe Ansatz equations \eqref{bethe} and the nondegenerate polynomial
solutions of the $qq$-system \eqref{qq}. \\
ii) If $\langle \alpha_{l},Z^H\rangle=0$, 
for $l=i_1,\dots, i_k$ and is nonzero otherwise, then
$\{q^i_+(z)\}_{i=1,\dots, r}$ and 
$\{q^i_-(z)\}_{i\neq {i_1,\dots, i_k}}$, are uniquely determined by
the Bethe Ansatz equations, but each $\{q^{i_j}_-(z)\}$ for $j=1,\dots
k$ is only determined up to an arbitrary transformation $q^{i_j}_-(z)\to q^{i_j}_-(z)+c_jq^{i_j}_+(z)$, where $c_j\in \mathbb{C}$.
\end{Thm}  

From now on we will drop the superscript $H$ over $Z$ and consider $Z$ to be an element of Cartan subalgebra.  We also remark, that in the case of Gaudin model, one puts the restrictions on degrees of $\{q^i_+(z)\}_{i=1,\dots, r}$, so that they determine a certain weight: \begin{eqnarray}
\check{\Lambda}=\sum^N_{i=1}\check{\lambda}_i-\sum_i{\rm deg}(q^i_{+}(z))\check{\alpha}_i
\end{eqnarray}  
in the representation of $^L\mathfrak{g}$, and the degrees of $\{q^{i,w}_{+}\}_{w\in W,~i=1,\dots, r} $ in the full $qq$-system determine 
$w\cdot\check{\Lambda}=\sum^N_{i=1}\check{\lambda}_i-\sum_i{\rm deg}(q^{w,i}_{+}(z))\check{\alpha}_i$. This means that with this restriction on the degrees, the $qq$-system  can be extended to the 
full $qq$-system, allowing polynomial solutions. From now on we will assume that $qq$-systems allow such extension to the full $qq$-system.

The Miura $G$-oper connection operator, which correspond to the $qq$-system can be explicitly written as follows:
\begin{eqnarray}\label{stform}
\nabla=\partial_z-Z+\sum^r_{i=1}\partial_z\log(q^i_+(z))\check{\alpha}_i+\sum^r_{i=1}\Lambda_i(z)f_i~.
\end{eqnarray}

Let us explain what role the full $qq$-system will play in this context. To do that, we use the following Proposition to introduced {\it B\"acklund transformations}, aligned with the action of the generators of the Weyl group \footnote{ We note here that the degenerate version of the $qq$-system ($Z$=0), as well as degenerate version of the Proposition below were introduced by Mukhin and Varchenko \cite{mukhvarmiura}.}.

\begin{Prop} \cite{Brinson:2021ww}   \label{fiter}  Let $\{q^j_+,q^j_-\}_{j=1,\dots,r}$ 
be a polynomial solution of the $qq$-system \eqref{qq}, and let $\nabla$ be the connection  in the form \eqref{stform}. Let $\nabla^{(i)}$ be the connection obtained
  from $\nabla$ via the gauge transformation by $e^{\mu_i(z)e_i}$,
  where
  \begin{equation}\label{eq:PropDef}
    \mu_i(z)=\Lambda_i(z)^{-1}\Bigg[\partial_z\log\Bigg(\frac{q^i_-(z)}{q^i_+(z)}\Bigg)+\langle
    \alpha_i,Z^H\rangle\Bigg]
  \end{equation}
Then $\nabla^{(i)}$ is obtained by making the following substitutions
in \eqref{stform}:

\begin{equation}\begin{aligned}
q^j_+(z) &\mapsto q^j_+(z), \qquad j \neq i, \\
q^i_+(z) &\mapsto q^i_-(z), \qquad Z\mapsto s_i(Z^H)=Z^H-\langle \alpha_i, Z^H\rangle\ \check{\alpha}_i~.
\label{eq:Aconnswapped}  
\end{aligned}
\end{equation}
\end{Prop}   
Thus the sequence of B\"acklund transformations produces $Z^w$-twisted Miura $G$-oper, where
\begin{equation}
Z^w=\tilde{w}Z\tilde{w}^{-1},\quad w\in W, 
\end{equation}
 corresponding to a given $G$-oper, each of which is constructed from $q^{i,w}_{\pm}(z)$.

\section{From $Z$-twisted Miura $G$-opers to $G$-Wronskians}

We see that there is a crtain relation between nondegenerate $G$-Wronskians and 
nondegenerate $G$-opers. Let us make this correspondence explicit. The $Z$-twisted condition implies that that there exist an element $\mathscr{B}_-(z)$, such that 
\begin{eqnarray}\label{mopdiag}
\nabla =\mathcal{B}^{-1}_-(z)(\partial_z-Z)\mathcal{B}_-(z), 
\end{eqnarray}
where 
\begin{equation}    \label{vdots}
\mathcal{B}_-(z) = \prod_{i=1}^r
e^{\frac{q^i_{-}(z)}{q^i_{+}(z)} f_i}\prod_{i=1}^r \Big[q_+^i(z)\Big]^{\check\alpha_i}\dots,
\end{equation}
where dots stand for the terms from $[N_-,N_-](z)$.
Let us use the following upper-triangular transformation, choosing a certain order in the product below: 
\begin{equation}\label{upper}
\mathcal{N}_+(z)=\prod^r_{i=1}e^{\frac{(\partial_z\log(q_+^i(z))-\zeta_i)e_i}{\Lambda_i(z)}}
\end{equation} 
such that 
\begin{equation}
\mathcal{N}^{-1}_+(z)\nabla\mathcal{N}_+(z)=\partial_z+p^{\Lambda}_{-1}+n_+(z)
\end{equation}
where $n_+(z)\in \mathfrak{n}_+(z)$. 
Then applying  this operator to the highest weight vector $\nu$ in some representation $V$ the product $\mathcal{G}(z)=\mathcal{B}_-(z)\mathcal{N_+}(z)$ satisfies the equation
\begin{eqnarray}
\mathcal{G}(z)^{-1}(\partial_z-Z)\mathcal{G}(z)\nu=p^{\Lambda}_{-1}\nu,
\end{eqnarray}
which coincides with equations $\ref{eqnreg}$ when $\nu$ varies over highest weights corresponding to fundamental representations. Turns out this correspondence works in a different direction as well, namely we want to show the following.


\begin{Thm}
There is a one-to-one correspondence between nondegnerate $Z$-twisted Miura G-opers, and equivalence classes of $G$-Wronskians, corresponding to the solution of the nondegenerate full $qq$-system. The element $\mathcal{G}(z)=\mathcal{B}_-(z)\mathcal{N}_+(z)\in G(z)$, where $\mathscr{B}_-(z)$, $\mathscr{N}_+(z)$ are defined in \eqref{vdots}, \eqref{upper} correspondingly, is a representative in the class of $G$-Wronskians corresponding to a given solution of the full $qq$-system.
\end{Thm}

\begin{proof}
Given $\mathscr{G}(z)=\mathscr{B}_-(z)\mathscr{N}_+(z)\in G(z)$ corresponding to a $G$-Wronskian, we have the following equations:
\begin{eqnarray}
\mathscr{N}_+(z)^{-1}\mathscr{B}_-(z)^{-1}(\partial_z-Z)\mathscr{B}_-(z)\mathscr{N}_+(z)\nu_i=p^{\Lambda}_{-1}\nu_i, \quad, i=1,\dots , r.
\end{eqnarray}
This means $\mathscr{N}_+(z)^{-1}A(z)\mathscr{N}_+(z)\nu_i=p^{\Lambda}_{-1}\nu_i$, where 
\begin{eqnarray}
\partial_z+A(z)=\mathscr{B}_-(z)^{-1}(\partial_z-Z)\mathscr{B}_-(z),
\end{eqnarray}
so that $A(z)\in \mathfrak{b}_-(z)$. Then  $A(z)\in p^{\Lambda}_{-1}+\mathfrak{h}_+(z)$, namely a Miura oper connection. That, however, implies that $A(z)$ defines a Miura oper connection. The diagonal part of $\mathscr{B}_-(z)$ is given by principal minors, which are exactly the $q^i_{+}(z)$, thus reproducing the diagonal part of $Z$-twisted Miura-Plucker oper.

\end{proof}

At the same time, the nondegeneracy conditions on $qq$-systems is an open condition, which was only needed for Miura $G$-oper connection itself and the corresponding $Z$-twisted condition, which is not the case for $G$-Wronskians, which only care about the operator (\ref{stform}), which satisfies $Z$-twisted condition as long as $\{q_+^i(z)\}_{i=1,\dots,r}$ satisfy the $qq$-system.  Therefore we have the following theorem.
\begin{Thm}
There is a one-to-one correspondence between solutions of the full $qq$-system, such that $q^{i,w}_+(z)\neq 0$ and equivalence classes of nondegenerate $G$-Wronskians.  
\end{Thm}

\subsection*{Example: $SL(2)$-opers and $SL(2)$-Wronskians} 
The $Z$-twisted Miura $SL(2)$-oper connection is given by the following matrix operator in defining representation of $\mathfrak{sl(2)}$:
\begin{equation}
\partial_z+
\begin{pmatrix}
\partial_z\log(q_+)-\zeta &\Lambda(z)\\
0&\zeta-\partial_z\log(q_+)
\end{pmatrix}
\end{equation}
where $\Lambda(z)$ is a polynomial, defining regular singularities $q_+(z)$ is a polynomial, satisfying $qq$-system:
\begin{equation}\label{sl2qq}
q_+(z)\partial_zq_-(z)-q_-(z)\partial_zq_+(z)+2\zeta q_+(z)q_-(z)=\Lambda(z)
\end{equation}
The matrices corresponding to $\mathcal{B}_+(z)$ and $\mathcal{N}_+$ are represented by the following matrices:
\begin{equation}
\mathcal{B}_-(z)=
\begin{pmatrix}
q_+(z)&0\\
q_-(z)&q_+(z)^{-1}
\end{pmatrix}
\quad 
\mathcal{N}_+(z)=
\begin{pmatrix}
1&\Lambda(z)^{-1}(\partial_z\log[q_+(z)]-\zeta)\\
0&1
\end{pmatrix},
\end{equation}
so that their product give a version of a Wronskian (one needs to use the qq-system (\ref{sl2qq}) to obtain the answer):
\begin{equation}
\mathcal{G}(z)=\mathcal{B}_-(z)\mathcal{N}_+(z)=
\begin{pmatrix}
q_+(z)&\Lambda^{-1}(\partial_z-\zeta)q_+(z)\\
q_-(z)&\Lambda^{-1}(\partial_z+\zeta)q_-(z)
\end{pmatrix}.
\end{equation}

\subsection{Comparison to $(G,q)$-oper case} In \cite{KZminors} we considered a similar construction related to $(G,q)$-opers. The $(G,q)$-Wronskian equations, which is the $q$-difference version of (\ref{eqnreg}) is as follows:
\begin{eqnarray}\label{qwr}
Z^{-1}\mathscr{G}(qz)\nu_i=\mathscr{G}(z)s_{\Lambda}^{-1}(z)\nu_i, \quad i=1,\dots r,
\end{eqnarray}
where $\mathscr{G}(z)\in G(z)$, $Z\in H$, $s_{\Lambda}(z)=\prod^r_{i=1}\Lambda_i^{\check{\alpha}_i}(z)s_i$ is a lift of a chosen Coxeter element to $G(z)$ and $\Lambda_i(z)$ are polynomials, corresponding to regular singularities.
As in the differential case, one can define a class of those: namely, if $\mathscr{G}(z)$ is a solution of (\ref{qwr}), then so is $\mathscr{G}(z)n_+(z)$, as long as  $s_{\Lambda}n_+(z)s^{-1}_{\Lambda}\in N_+(z)$. 
Choosing Coxeter element in such a way that $s_{\Lambda}^{h/2}=w^{\Lambda}_0$, where $w^{\Lambda}_0$ is a lift of the Weyl reflection, corresponding to the longest root to  to $G(z)$, we can iterate the equations $\ref{qwr}$:
\begin{eqnarray}\label{qwrit}
&&Z^{-k}\mathscr{G}(q^kz)\nu_i=\mathscr{G}(z)s_{\Lambda}^{-1}(z)s_{\Lambda}^{-1}(qz)\dots s_{\Lambda}^{-1}(q^{k-1}z)\nu_i,\nonumber \\
&& i=1,\dots r, \quad k=1, \dots h/2.
\end{eqnarray}
The solution of this kind picks up a unique element in the equivalence class of $(G,q)$-Wronskians, moreover, we were able to write down a universal formula for it. In the $SL(N)$ case that gives a suitably twisted $q$-Wronskian matrix. Unfortunately it is non-obvious whether if is possible to write a similar universal formula in differential case: the blunt approach using differential analogue of equations (\ref{qwrit}) do not seem to work beyond defining representation of $SL(N)$.

\bibliography{cpn1}

\begin{bibdiv}
\begin{biblist}

\bib{Aganagic:2017be}{article}{
      author={Aganagic, Mina},
      author={Okounkov, Andrei},
       title={{Quasimap counts and Bethe eigenfunctions}},
        date={2017},
     journal={Moscow Math. J.},
      volume={17},
      number={4},
       pages={565\ndash 600},
      eprint={1704.08746},
}

\bib{Baxter:1982zz}{book}{
      author={Baxter, R.~J.},
       title={{Exactly solved models in statistical mechanics}},
   publisher={Academic Press},
        date={1982},
        ISBN={0486462714, 9780486462714},
         url={http://www.amazon.com/dp/0486462714},
}

\bib{Bethe:1931hc}{article}{
      author={Bethe, H.},
       title={{Zur Theorie der Metalle I. Eigenwerte und Eigenfunktionen der
  linearen Atomkette}},
        date={1931},
     journal={Z. Phys.},
      volume={71},
       pages={205\ndash 226},
}

\bib{Beilinson:2005}{article}{
      author={Beilinson, A.},
      author={Drinfeld, V.},
       title={{Opers}},
        date={2005},
      eprint={math/0501398},
         url={https://arxiv.org/abs/math/0501398},
}

\bib{Bazhanov:1998dq}{article}{
      author={Bazhanov, Vladimir~V.},
      author={Lukyanov, Sergei~L.},
      author={Zamolodchikov, Alexander~B.},
       title={{Integrable Structure of Conformal Field Theory. 3. The
  Yang-Baxter Relation}},
        date={1999},
     journal={Commun. Math. Phys.},
      volume={200},
       pages={297\ndash 324},
      eprint={hep-th/9805008},
}

\bib{Brinson:2021ww}{article}{
      author={Brinson, Ty~J.},
      author={Sage, Daniel~S.},
      author={Zeitlin, Anton~M.},
       title={Opers on the projective line, wronskian relations, and the bethe
  ansatz},
        date={202112},
      eprint={2112.02711},
         url={https://arxiv.org/pdf/2112.02711.pdf},
}

\bib{CG}{book}{
      author={Chriss, N.},
      author={Ginzburg, V.},
       title={{Representation Theory and Complex Geometry}},
   publisher={Birkh\"{a}user Boston, Inc., Boston, MA},
        date={1997},
        ISBN={0-8176-3792-3},
}

\bib{Frenkel:2003qx}{article}{
      author={Frenkel, E.},
       title={{Opers on the Projective Line, flag Manifolds and Bethe Ansatz}},
        date={2003},
     journal={Mosc. Math. J.},
      volume={4},
       pages={655\ndash 705},
      eprint={math/0308269},
}

\bib{Frenkel:2004qy}{article}{
      author={Frenkel, E.},
       title={{Gaudin Model and Opers}},
        date={2004},
     journal={{Infinite Dimensional Algebras and Quantum Integrable Systems,
  eds. P. Kulish, e.a., Progress in Math.}},
      volume={237},
       pages={1\ndash 60},
      eprint={math/0407524},
}

\bib{Frenkel_LanglandsLoop}{book}{
      author={Frenkel, Edward},
       title={Langlands correspondence for loop groups},
   publisher={Cambridge University Press},
        date={2007},
}

\bib{Feigin:1994in}{article}{
      author={Feigin, B.},
      author={Frenkel, E.},
      author={Reshetikhin, N.},
       title={{Gaudin Model, Bethe Ansatz and Critical Level}},
        date={1994},
     journal={Commun. Math. Phys.},
      volume={166},
       pages={27\ndash 62},
      eprint={hep-th/9402022},
}

\bib{Rybnikov:2010}{article}{
      author={Feigin, B.},
      author={Frenkel, E.},
      author={Rybnikov, L.},
       title={{Opers with irregular singularity and spectra of the shift of
  argument subalgebra}},
        date={2010},
     journal={Duke Math. J.},
      volume={155},
       pages={337\ndash 363},
}

\bib{Feigin:2006xs}{article}{
      author={Feigin, B.},
      author={Frenkel, E.},
      author={Toledano~Laredo, V.},
       title={{Gaudin models with irregular singularities}},
        date={2010},
     journal={Adv. Math.},
      volume={223},
       pages={873\ndash 948},
      eprint={math/0612798},
}

\bib{Frenkel:2013uda}{article}{
      author={Frenkel, E.},
      author={Hernandez, D.},
       title={{Baxter's relations and spectra of quantum integrable models}},
        date={2015},
     journal={Duke Math. J.},
      volume={164},
      number={12},
       pages={2407\ndash 2460},
      eprint={1308.3444},
}

\bib{Frenkel:2016}{article}{
      author={Frenkel, E.},
      author={Hernandez, D.},
       title={{Spectra of Quantum KdV Hamiltonians, Langlands Duality, and
  Affine Opers}},
        date={2018},
     journal={Commun. Math. Phys.},
      volume={362},
       pages={361\ndash 414},
      eprint={1606.05301},
         url={https://arxiv.org/abs/1606.05301},
}

\bib{FHRnew}{article}{
      author={Frenkel, E.},
      author={Hernandez, D.},
      author={Reshetikhin, N.},
       title={{Folded quantum integrable models and deformed $W$-algebras}},
     journal={Lett. Math. Phys., to appear},
      eprint={2110.14600},
         url={http://https://arxiv.org/abs/2110.14600},
}

\bib{Frenkel:2020}{article}{
      author={Frenkel, Edward},
      author={Koroteev, Peter},
      author={Sage, Daniel~S.},
      author={Zeitlin, Anton~M.},
       title={{q-Opers, QQ-Systems, and Bethe Ansatz}},
        date={2020},
     journal={J. Europ. Math. Soc.},
       pages={to appear},
      eprint={2002.07344},
         url={https://arxiv.org/pdf/2002.07344.pdf},
}

\bib{FZ}{article}{
      author={Fomin, S.},
      author={Zelevinsky, A.},
       title={{Double Bruhat Cells and Total Positivity}},
        date={1999},
        ISSN={0894-0347},
     journal={J. Amer. Math. Soc.},
      volume={12},
      number={2},
       pages={335\ndash 380},
      eprint={math/9802056},
  url={https://doi-org.libezp.lib.lsu.edu/10.1090/S0894-0347-99-00295-7},
}

\bib{Fomin:wy}{article}{
      author={Fomin, Sergey},
      author={Zelevinsky, Andrei},
       title={{Recognizing Schubert Cells}},
        date={2000},
     journal={J. Algeb. Comb.},
      volume={12},
       pages={37\ndash 57},
      eprint={math/9807079},
         url={https://arxiv.org/pdf/math/9807079.pdf},
}

\bib{HJ}{article}{
      author={Hernandez, D.},
      author={Jimbo, M.},
       title={{Asymptotic Representations and Drinfeld Rational Fractions}},
        date={2012},
     journal={Comp. Math.},
      volume={148},
      number={5},
       pages={1593\ndash 1623},
      eprint={1104.1891},
}

\bib{Korepin_1993}{article}{
      author={Korepin, V.},
      author={Bogoliubov, N.},
      author={Izergin, A.},
       title={{Quantum Inverse Scattering Method and Correlation Functions}},
        date={1993},
     journal={Cambridge University Press},
         url={http://dx.doi.org/10.1017/CBO9780511628832},
}

\bib{Koroteev:2017aa}{article}{
      author={Koroteev, P.},
      author={Pushkar, P.},
      author={Smirnov, A.},
      author={Zeitlin, A.},
       title={{Quantum {K}-theory of Quiver Varieties and Many-Body Systems}},
        date={2017},
     journal={Selecta Math. in press},
      eprint={1705.10419},
         url={https://arxiv.org/pdf/1705.10419},
}

\bib{KSZ}{article}{
      author={Koroteev, P.},
      author={Sage, D.},
      author={Zeitlin, A.},
       title={{(SL(N),q)-Opers, the q-Langlands Correspondence, and
  Quantum/Classical Duality}},
        date={2021},
     journal={Commun. Math. Phys.},
      volume={381},
      number={2},
       pages={641\ndash 672},
      eprint={1811.09937},
         url={https://arxiv.org/abs/1508.02690},
}

\bib{KZminors}{article}{
      author={Koroteev, P.},
      author={Zeitlin, A.M.},
       title={{q-Opers, QQ-systems, and Bethe Ansatz II: Generalized Minors}},
      eprint={2108.04184},
         url={https://arxiv.org/abs/2108.04184},
}

\bib{Koroteev:2020tv}{article}{
      author={Koroteev, Peter},
      author={Zeitlin, Anton~M.},
       title={{Toroidal q-Opers}},
        date={2020},
     journal={J. Inst. Math. Jussieu, in press},
      eprint={2007.11786},
         url={https://arxiv.org/pdf/2007.11786.pdf},
}

\bib{mukhvarmiura}{article}{
      author={Mukhin, E.},
      author={Varchenko, A.},
       title={{Miura Opers and Critical Points of Master Functions}},
     journal={Cent. Eur. J. Math.},
      volume={3},
      eprint={math/0312406},
         url={https://arxiv.org/abs/math/0312406},
}

\bib{Nekrasov:2009ui}{article}{
      author={Nekrasov, N.},
      author={Shatashvili, S.},
       title={{Quantum integrability and supersymmetric vacua}},
        date={2009},
     journal={Prog. Theor. Phys. Suppl.},
      volume={177},
       pages={105\ndash 119},
      eprint={0901.4748},
}

\bib{Nekrasov:2009uh}{article}{
      author={Nekrasov, N.},
      author={Shatashvili, S.},
       title={{Supersymmetric vacua and Bethe ansatz}},
        date={2009},
     journal={Nucl. Phys. Proc. Suppl.},
      volume={192-193},
       pages={91\ndash 112},
      eprint={0901.4744},
}

\bib{Okounkov:2015aa}{article}{
      author={Okounkov, A.},
       title={{Lectures on {K}-theoretic computations in enumerative
  geometry}},
        date={2015},
      eprint={1512.07363},
         url={https://arxiv.org/abs/1512.07363},
}

\bib{OGIEVETSKY1986360}{article}{
      author={Ogievetsky, E.},
      author={Wiegmann, P.},
       title={{Factorized S-matrix and the Bethe ansatz for Simple Lie
  Groups}},
        date={1986},
        ISSN={0370-2693},
     journal={Physics Letters B},
      volume={168},
      number={4},
       pages={360 \ndash  366},
  url={http://www.sciencedirect.com/science/article/pii/0370269386916448},
}

\bib{Pushkar:2016qvw}{article}{
      author={Pushkar, Petr},
      author={Smirnov, Andrey},
      author={Zeitlin, Anton},
       title={{Baxter Q-operator from quantum K-theory}},
        date={2020},
     journal={Adv. Math.},
      volume={360},
       pages={106919},
      eprint={1612.08723},
}

\bib{Reshetikhin:2010si}{article}{
      author={Reshetikhin, N.},
       title={Lectures on the integrability of the 6-vertex model},
        date={2010},
      eprint={1010.5031},
         url={https://arxiv.org/abs/1010.5031},
}

\bib{Takhtajan:1979iv}{article}{
      author={Takhtajan, L.~A.},
      author={Faddeev, L.~D.},
       title={{The Quantum method of the inverse problem and the Heisenberg XYZ
  model}},
        date={1979},
     journal={Russ. Math. Surveys},
      volume={34},
      number={5},
       pages={11\ndash 68},
        note={[Usp. Mat. Nauk34,no.5,13(1979)]},
}

\end{biblist}
\end{bibdiv}

\end{document}